\documentclass[a4paper, 12pt]{article}

\usepackage{amsfonts}
\usepackage {amssymb}
\usepackage {amsmath}
\usepackage {amsmath}
\usepackage {amsthm}
\usepackage{graphicx}
\usepackage{multirow}
\usepackage{xypic}
\usepackage {amscd}
\usepackage{mathrsfs}
\usepackage[colorlinks, linkcolor=blue, anchorcolor=black, citecolor=red]{hyperref}
\usepackage{enumerate}

\usepackage{geometry}  

\newcommand{\bC}{\mathbb{C}}
\newcommand{\bQ}{\mathbb{Q}}

\newcommand{\ol}{\overline}
\newcommand{\vphi}{\varphi}
\newcommand{\la}{\langle}
\newcommand{\ra}{\rangle}
\newcommand{\iz}{\mathrm{Im}(z)}

\newcommand{\bP}{\mathbb{P}}
\newcommand{\cY}{\mathcal{Y}}
\newcommand{\cO}{\mathcal{O}}
\newcommand{\mcH}{\mathcal{H}}

\newcommand{\cX}{\mathcal{X}}
\newcommand{\cF}{\mathcal{F}}

\newtheorem{thm}{Theorem}[section]

\newtheorem{defn}[thm]{Definition}

\newtheorem{rem}[thm]{Remark}

\newtheorem{lem}[thm]{Lemma}

\newtheorem{defn-prop}[thm]{Definition-Proposition}

\begin{document}

\title{Polarized Hodge Structures for Clemens Manifolds}
\author{Chi Li}
\date{}

\maketitle

\abstract{
Let $X$ be a Calabi-Yau threefold. A conifold transition first contracts $X$ along disjoint rational curves with normal bundles of type $(-1,-1)$, and then smooth the resulting singular complex space $\bar{X}$ to a new compact complex manifold $Y$. Such $Y$ is called a Clemens manifold and can be non-K\"{a}hler. 
We prove that any small smoothing $Y$ of $\bar{X}$ satisfies $\partial\bar{\partial}$-lemma. We also show that the resulting pure Hodge structure of weight three on $H^3(Y)$ is polarized by the cup product. 
These results answer some questions of R. Friedman.  }

\tableofcontents

\section{Introduction and main results}

Let $X$ be a projective Calabi-Yau threefold. Let $C_1,\dots, C_r$ be $r$ disjoint smooth rational curves such that their normal bundles $N_{C_i}X\cong \mathcal{O}_{\bP^1}(-1,-1)$. Then they can be contracted to $r$ ordinary double points to give a singular complex space $\bar{X}$. 
If the class $[C_1], \dots, [C_r]$ span $H^4(X;\bC)$ and satisfy a linear relation $\sum_i m_i [C_i]=0$ with all $m_i\neq 0$, then $\bar{X}$ is smoothable and all smoothings of $\bar{X}$ are compact complex manifolds of dimension three with second Betti number $b_2=0$ and trivial canonical bundle. 
Such non-K\"{a}hler complex threefolds were first constructed by H. Clemens and are thus called Clemens manifolds. The details of this construction appeared in R. Friedman \cite{Fri86} under an assumption on the obstruction group of deformation, which was later removed in \cite{Kaw92, Ran92, Tia92}. The transition from $X$ to smoothings of $\bar{X}$ is called the conifold transition, which is a process for connecting different moduli spaces of projective Calabi-Yau threefolds (see \cite{Fri86, Rei87}). For this reason, Clemens manifolds have been studied extensively in complex geometry (see \cite{Yau21}). 

For a given Clemens manifold $Y$, it is a natural question whether $Y$ satisfies the $\partial\bar{\partial}$-lemma.
According to the work \cite{DGMS}, this is equivalent to asking whether the Hodge-de Rham spectral sequence degenerates at the $E_1$-page and 
whether Hodge filtrations on the cohomology groups induce pure Hodge structures. We can further ask whether the possible induced weight 3 Hodge structure on $H^3(Y)$ is polarized by the cup product. As pointed out by Friedman in \cite{Fri20}, Clements manifolds are not deformation equivalent to complex manifolds bimeromorphic to K\"{a}hler manifolds, so that the standard method for endowing (polarized) Hodge structures on cohomology groups do not apply. Despite this fact, Friedman proved in \cite{Fri20} that the general Clemens manifolds satisfy the $\partial\bar{\partial}$-lemma.
Here general roughly means that the $\partial\bar{\partial}$-lemma holds outside of a proper real analytic subvariety in the moduli space. Note that satisfying $\partial\bar{\partial}$-lemma is not a closed condition under deformation (see \cite{Ang14} for concrete examples of nilmanifolds). 
Furthermore, Friedman wrote that 
\vskip 2mm
\noindent
\textit{
``... although it seems likely that in fact it holds for all small smoothings of $\bar{X}$. Unfortunately, the variational methods of this paper do not seem well suited to deciding if the resulting weight three Hodge structures are polarized."}
\vskip 2mm

The goal of this note is to confirm and answer the above speculation:

\begin{thm}\label{thm-ddbar}
Any small smoothing $Y$ of $\bar{X}$ satisfies the $\partial\bar{\partial}$-lemma. Moreover, the resulting weight three Hodge structure on $H^3(Y)$ is polarized by the cup product. 
\end{thm}

Let $Y$ be a smoothing of $\bar{X}$ and let $\{F^p\}$ be the Hodge filtration of $H^3(Y; \bC)$. Then based on the deep theory of the (limiting) mixed Hodge structures, Friedman already showed in \cite{Fri20} that the Hodge-de Rham spectral sequence degenerates at the $E_1$ page, which implies $F^1\cap \ol{F^3}=0$ (\cite[Lemma 1.5]{Fri20}) and $H^{k}(Y; \bC)$ trivially admits a pure Hodge structure if $k\neq 3$. As a consequence, as explained in \cite[Corollary 1.6]{Fri20}, it follows from the work of Deligne-Griffiths-Morgan-Sullivan \cite{DGMS} that $Y$ satisfies the $\partial\bar{\partial}$-lemma if and only if $F^2\cap \ol{F^2}=0$. Recall that a pure Hodge structure on $H^3(Y; \bC)$ is ($\bQ$)-polarized if there is non-degenerate bilinear form (defined over $\bQ$) that satisfies Hodge-Riemann (HR) bilinear relation. 
Friedman showed that, for the resulting pure Hodge structure on the general Clemens manifolds, the first set of HR bilinear relation follows from the degeneration of Hodge-de Rham spectral sequence at the $E_1$-page (\cite[Lemma 1.4]{Fri20}):
\begin{equation}\label{eq-HR1}
\la F^1, F^3\ra=0, \quad \la F^2, F^2\ra=0
\end{equation}
where $\la\cdot, \cdot\ra$ is the standard (alternating) intersection form on $H^3(Y; \bQ)$. 
Moreover the (easy) $H^{3,0}$-part of the second second set of HR bilinear relation is true by a direct calculation:
$
\sqrt{-1}\la \omega, \bar{\omega}\ra>0,  
$
for any $\omega\in H^{3,0}$. 

Given these discussion, the precise statement of our result is the following:
\begin{thm}\label{thm-HR}
Let $Y$ be any small smoothing of $\bar{X}$. 
The subspace $F^2_Y=F^2 H^3(Y; \bC)$ satisfies:
\begin{equation}\label{eq-opposed}
F^2_Y\cap \ol{F^2_Y}=0.
\end{equation}
We have a Hodge decomposition $H^3(Y;\bC)=H^{3,0}\oplus H^{0,3}\oplus H^{2,1}\oplus H^{1,2}$ where $H^{p,q}=F^p_Y\cap \ol{F^q_Y}$ satisfies $H^{q,p}=\ol{H^{p,q}}$. 
Moreover, the following Hodge-Riemann bilinear relation is true: for any $\eta\in H^{2,1}$, 
\begin{equation}\label{eq-HR3}
-\sqrt{-1}\la \eta, \bar{\eta}\ra>0.
\end{equation}
\end{thm}

The proof of Theorem \ref{thm-HR} uses Friedman's calculation of limiting mixed Hodge structure $H^3_{\lim}$ associated to certain semistable model of degeneration (as stated in Theorem \ref{thm-LMH} below). For the proof of \eqref{eq-opposed}, some part of our calculations is motivated by the calculation in \cite{Fri20}. However our use of the limiting mixed Hodge structure is different (see Remark \ref{rem-choice}). The extra tool we use is Deligne's (weak) splitting of mixed Hodge structures (see Theorem \ref{thm-splitting}). 
The key new observation and an interesting phenomenon (based on elementary linear algebra) is that, the invariance of certain well-chosen basis sections of $F^2$ under the monodromy operator implies that their wedge product has a leading term given by $\iz\cdot |D|^2$ where $z$ is in the upper half-plane such that $\iz\gg 1$ when the smoothing is small and $D$ is a smooth function on the base of deformation which is bounded away from zero (see \eqref{eq-leading1}). 
The proof of Hodge-Riemann bilinear relation \eqref{eq-HR3} is more delicate. 
We first observe that $\{H^{2,1}_s=F^2_{Y_s}\cap \ol{F^1_{Y_s}}; s\in S\}$, where $S$ is the base of deformations, form a sub-vector bundle which allows us to extend a well-chosen basis from $H^{2,1}_{\lim}$ and consider intersection pairing on $Y$ as a block matrix. The restriction of intersection pairing to a one-dimensional subspace can be shown to be positive by a similar technique as before. We then reduce the problem to proving the positive definiteness of a sub-matrix on the complement after a suitable change of basis. 
The key is to show that this sub-matrix is essentially a perturbation of the intersection pairing on $X$ (Lemma \ref{lem-transit}), which allows to transmit the Hodge-Riemann relation on $H^3(X)$ to the positive definiteness of this sub-matrix. The method in this paper should also work for some other examples of non-K\"{a}hler manifolds obtained in similar smoothing process, and this will be studied elsewhere. 

\textbf{Acknowledgement:} 
The author is partially supported by NSF (Grant No. DMS-1810867) and an Alfred P. Sloan research fellowship.
The author thanks R. Friedman and G. Tian for their interest. He also thanks G. Tian and S. Rao for helpful comments. 

\section{Data of the limiting mixed Hodge structure}

For the reader's convenience and for comparison to the argument in \cite{Fri20}, we will moslty use similar notation from  loc. cit. In particular we set $h:=\dim H^1(X; \Omega^2_X)=h^{2,1}(X)=h^{1,2}(X)$ so that $\dim H^3(X; \bC)=2h+2$. Recall that we assume that the classes $\{[C_i]; i=1\dots, r\}$ span the $H^4(X;\bC)$ and there exists a linear relation $\sum_i m_i [C_i]=0$ with $m_i\neq 0$. 
According to \cite[Remark 2.3]{Fri20}, we can and will assume that the no $r-1$ elements of $\{[C_i]; i=1,\dots, r\}$ is linearly dependent. In particular we have $b_2(X)=b_4(X)=r-1$. 

Let $\bar{\mu}: X\rightarrow \bar{X}$ be the contraction of all $C_i$'s. 
Set $\mathbb{T}^i_{\bar{X}}=\mathrm{Ext}^1(\Omega^1_{\bar{X}},\cO_{\bar{X}})$. 
By \cite[Theorem 4.4]{Fri86} there is an exact sequence:
\begin{equation*}
0\rightarrow H^1(\bar{X}, T^0_{\bar{X}})\rightarrow \mathbb{T}^1_{\bar{X}}\rightarrow \bC\rightarrow 0
\end{equation*}
where the last term $\bC$ is identified with the kernel of the fundamental class map
$
\bigoplus_i \bC[C_i]\rightarrow H^4(X_0; \bC). 
$
Let $\bar{\pi}: \cX\rightarrow \bar{S}$ be the germ of the locally semi-universal deformation of $\bar{X}$. We can identify $\bar{S}$ with the germ around the origin in $\mathbb{T}^1_{\bar{X}}$. Note that in the current setting, we have the isomorphism $H^1(\bar{X}; T^0_{\bar{X}})\cong H^1(\bar{X}; R^0\bar{\mu}_*\Theta_X)=H^1(X; \Theta_X)\cong H^1(X; \Omega_X^2)$ (see \cite[Lemma 8.6]{Fri91}). In particular, $\dim \bar{S}=h+1$. The fiber $X_{\bar{s}}$ over $\bar{s}\in \bar{S}\setminus H^1(\bar{X}; T^0_{\bar{X}})$ is a smoothing of $\bar{X}$. Topologically it is obtained from $X$ by doing surgery which replaces $r$ copies of $S^2\times D^3$ (around $C_i$) by $r$ copies of $S^3\times D^3$. A calculation based on Mayer-Vietoris sequence shows that $b_3(X_{\bar{s}})=2h+4$, $b_2(X_{\bar{s}})=b_4(X_{\bar{s}})=0$ (see \cite[Theorem 3.3]{Ros06}). In particular we know that $X_{\bar{s}}$ is non-K\"{a}hler.  It is known that the complex manifold $X_{\bar{s}}$ is diffeomorphic to a connected sum of $h+2$ copies of $S^3\times S^3$ and any $h\ge 0$ can be attained (\cite{Fri91} and \cite{LT93}). 

There is a normal crossing model for the smoothing of $\bar{X}$ constructed by Friedman \cite{Fri20} as follows. 
 Let $S\rightarrow \bar{S}$ be the double cover of $\bar{S}$ branched along the smooth hypersurface $\bar{S}\cap H^1(\bar{X}, T^0_{\bar{X}})$ and let $\bar{\cY}\rightarrow S$ be the pull back family. Let $D$ be the inverse image of $H^1(\bar{X}, T^0_{\bar{X}})$. The fibres of $\bar{\cY}$ over $D$ have $r$ ordinary double points and the singularities of the of $\bar{\cY}$ are locally analytically isomorphic to products of ordinary double points with $D$. Blowing up these singular points gives a proper flat morphism $\pi: \cY\rightarrow S$, where $\cY$ is smooth and the discriminant locus of $\pi$ is $D$. 
 
 The central fibre $Y_0$ can be described as follows. 
First let $\mu: \tilde{X}\rightarrow X$ be the blowup of $X$ along $C_i, i=1,\dots, r$ with the exceptional divisor $\sqcup_i E_i$. 
Since $N_{C_i}X\cong \cO_{\bP^1}(-1,-1)$, we know that $E_i=\bP(N_{C_i}X)\cong  \bP^1\times\bP^1$. Moreover the natural map $H^3(X)\rightarrow H^3(\tilde{X})$ becomes an isomorphism of Hodge structures. 
Let $Q_i$ be the smooth quadric hypersurface in $\bP^4$ and $E_i$ be a hyperplane section of $Q_i$. Then $Y_0=\tilde{X}\sqcup \coprod_i Q_i/\sim$ where the equivalence relation $\sim$ means that we glue $E_i\subset Q_i$ to $E_i\subset \tilde{X}$. 

The fibers of $\pi$ over $D$ are locally trivial deformations of $Y_0$. Let $\cY_D=\pi^{-1}(D)$. Then $\cY_D$ is a divisor with normal crossings in $\cY$. For $s\not\in D$, the fiber $Y_s$ of $\pi$ is identified with the corresponding smooth fibre $X_{\bar{s}}$ of $\bar{\pi}$ where $\bar{s}\in \bar{S}$ is the point lying under $s$. 

After shrinking, we can assume that $S$ is a polydisk $\Delta^{h+1}$.
Set $S^*=S\setminus D\cong \Delta^h\times\Delta^*$.  Let $\vphi: \widetilde{S^*}\cong \Delta^h\times \mathfrak{H}\rightarrow S^*$ be the universal covering of $S^*$, where $\mathfrak{H}=\{z\in \bC; \iz>0\}$ is the upper half plane. 
Set $\cY_\infty=\cY^*\times_{S^*}\widetilde{S^*}$. Then $H^3(\cY_\infty;\bC)\cong H^3(Y_s; \bC)$ for any $s\not\in D$.  
As explained in \cite{Fri20}, the work of Steenbrink (\cite{Ste75}), which gives a geometric realization of the Hodge theoretical construction in \cite{Sch73}, can be adapted to give a limiting mixed Hodge structure as stated in the next theorem. 
We refer to the book \cite{PS08} for background, definitions and deep results about (limiting) mixed Hodge structures. 
\begin{thm}[\cite{Fri20}]\label{thm-E1deg}
\begin{enumerate}[(i)]
\item
The hypercohomology $H:=\mathbb{H}^3(Y_0; \Omega^{\bullet}_{\cY/S}(\log \cY_D)|_{Y_0})$ is isomorphic to the cohomology $H^3(\cY_\infty; \bC)\cong H^3(Y_s; \bC)$. The sheaf $\mcH^3=\mathbb{R}^3\pi_*\Omega^\bullet_{\cY/S}(\log\cY_D)$ is locally free. $\mcH$ is Deligne's canonical extension of $\mcH^3|_{S^*}$. 
\item There is a mixed Hodge structure $(H_{\lim}, F_{\lim}^{\bullet}, W^{\lim}_{\bullet})$ on $H$. The spectral sequence with the $E_1$ page
\begin{equation*}
E^{p,q}_1=H^q(Y_0; \Omega^\bullet_{\cY/S}(\log \cY_D)|_{Y_0}) \Longrightarrow H^{p+q}_{\lim}
\end{equation*}
degenerates at $E_1$ and the corresponding filtration on $H^{p+q}_{\lim}$ is the Hodge filtration. 
\item Possibly after shrinking $S$, the spectral sequence of coherent sheaves on $S$ whose $E_1$-page is:
\begin{equation*}
E^{p,q}_1=R^q \pi_*\Omega^p_{\cY/S}(\log \cY_D) \Longrightarrow \mathbb{R}^{p+q}\pi_*\Omega^{\bullet}_{\cY/S}(\log \cY_D)
=\mcH^{p+q}
\end{equation*}
degenerates at $E_1$. 
For $s\in S^*$, the Hodge-de Rham spectral sequence for $Y_s$ degenerates at the $E_1$-page. 
The sheaves $R^q\pi_*\Omega^p_{\cY/S}(\log \cY_D)$ are locally free. 
Moreover, there is a filtration of $\mcH^{3}$ by holomorphic subbundles $\cF^\bullet$. 

\end{enumerate}
\end{thm}
Note that in the first statement, the restriction $\mcH|_{S^*}$ is a vector bundle associated to the local system and is endowed with the flat Gauss-Manin connection. Deligne's canonical extension $\mcH$ is trivialized by using the space of flat multi-section of $\mcH|_{S^*}$. Then pullback $\widetilde{\mcH}:=\vphi^*\mcH$ is a trivial holomorphic vector bundle on $\widetilde{S^*}$. The monodromy operator $T$ acts on $\widetilde{\mcH}$ and the fibre of $\mcH$ is the quotient of $\widetilde{\mcH}$ under the monodromy action. Set $N=\log T$ to be the associated nilpotent operator.
Any element $v\in H$ determines a holomorphic section $e^{zN}v$ of $\vphi^*\mcH$ which is also considered as a flat multi-section of $\mcH$. If $\sigma:=\{\sigma(s); s\in S^*\}$ is a smooth section of $\mcH|_{S^*}$, then 
$\tilde{\sigma}(t,z):=(\vphi^*\sigma)(t, z)=e^{zN} \sigma(t, \zeta)$ is a smooth $H$-valued function that satisfies the invariance property: $\tilde{\sigma}(z+1)=T\tilde{\sigma}(z)$ for any $z$ in the upper half plane. 

The holomorphic subbundle in Theorem \ref{thm-E1deg}.(iii) can be understood as follows. 
By Griffith's theory of the variational of Hodge structures, for any $p$, $\{ F^p H^3(Y_s; \bC); s\in S^*\}$ form a holomorphic vector bundle that will be denoted by $\cF^p|_{S^*}$.
Let $\check{D}$ be the compact subvariety of the product of Grassmannian varieties that parametrizes flags of ranks $(2h+3, h+2,1)$ in $H^3\cong \bC^{2h+4}$ that satisfies the first Hodge-Riemann bilinear relation \eqref{eq-HR1}. 
The period map $S\rightarrow \check{D}$ lifts to a holomorphic map $\tilde{g}^*: \widetilde{S^*}\rightarrow \check{D}$. 
Then $e^{-zN}\tilde{g}^*(z)$ is invariant under the monodromy action and induces the holomorphic mapping $G^*: S^*\rightarrow \check{D}$. 
The holomorphic subbundles $\cF^\bullet$ over $S$ extend $\cF^\bullet|_{S^*}$, and correspond to an extension of $G: S\rightarrow \check{D}$ of $G^*$ such that $\cF^\bullet$ are the pull back of universal bundles on the corresponding Grassmannian (see \cite{GrSc75}). 
According to this description, if we have a trivialization of $\cF^p$ (with rank $d_p$) over $S$ by a smooth frame $\{\sigma_1,\dots, \sigma_{d_p}\}$, $\sigma_i$ is thought as a linear combination of flat multi-sections, and for any $\tilde{s}=(t,z)\in \widetilde{S^*}=\Delta^h\times \mathfrak{H}$, the filtration $F^p_{\tilde{g}(\tilde{s})}$ of $\tilde{g}(\tilde{s})\in \check{D}$ is spanned by $\{(\vphi^*\sigma_1)(\tilde{s}),\dots, (\vphi^*\sigma_{d_p})(\tilde{s})\}$.

On the other hand, by the celebrated work of Deligne \cite{Del71}, there is a (functorial) mixed Hodge structure on $H^3(Y_0)$ (see \cite[Section 4]{GrSc75} for an easy construction). Moreover the specialization map $H^3(Y_0)\rightarrow H^3(\cY_\infty)\cong H^3(Y_s)$ with $s\in S^*$, which is induced by the composition
$Y_s\rightarrow \cY\rightarrow Y_0$, becomes a morphism of mixed Hodge structure (see \cite[Theorem 11.29]{PS08}). 

In the following theorem we collect the important data of the limiting mixed Hodge structure, which was calculated by Friedman. 
\begin{thm}[Friedman, \cite{Fri20}]\label{thm-LMH}
\begin{enumerate}[(i)]
\item 
As mixed Hodge structure over $\bQ$, $H^3(Y_0)$ is an extension of the pure Hodge structure on $H^3(\tilde{X})$ by a pure weight two piece $\cong \bQ(-1)$:
\begin{equation}
0\rightarrow \bQ(-1) \rightarrow H^3(Y_0)\rightarrow H^3(\tilde{X})\rightarrow 0.
\end{equation}
\item There is an exact sequence of mixed Hodge structures:
\begin{equation*}
0\rightarrow H^3(Y_0)\rightarrow H^3_{\lim}\rightarrow \bQ(-2)\rightarrow 0. 
\end{equation*}
The weight filtration $W_\bullet=W^{\lim}_\bullet$ on $H^3_{\lim}$ is given by:
\begin{equation*}
0\subseteq W_2\subseteq W_3\subset W_4=H^3_{\lim}
\end{equation*}
where $W_2\cong \bQ(-1)$, $W_3=H^3(Y_0)$, $W_3/W_2\cong H^3(\tilde{X})\cong H^3(X)$ and $W_4/W_3\cong \bQ(-2)$.  
\item
The Hodge filtration $F^\bullet:=F^\bullet_{\lim}$ satisfies:
\begin{enumerate}
\item $\dim F^1=2h+3$, $F^1\cap W_2=W_2$. 
\item $\dim F^2=h+2$, $F^2\cap W_2=0$ and $F^2+W_3=W_4$. 
\item $F^3\subseteq W_3$. 
\end{enumerate}
\item
The nilpotent operator $N: H^3_{\lim}\rightarrow H^3_{\lim}$ satisfies $N^2=0$ and:
\begin{equation*}
\mathrm{Ker}(N)=W_3, \quad \mathrm{Im}(N)=W_2. 
\end{equation*}
The induced map $N: W_3/W_2\rightarrow W_2$ is an isomorphism. 
\item 
 $W_3=W_2^\perp$ with respect to the intersection form on $H^3_{\lim}\cong H^3(\cY_\infty)=H^3(Y_s)$ with $s\in S^*$.

\end{enumerate}
\end{thm}

To continue, we will use the following splitting result by Deligne:
\begin{thm}[{see \cite[Lemma-Definition 3.4]{PS08}}]\label{thm-splitting}
For any mixed Hodge structure $(H, F^\bullet, W_\bullet)$, the bigrading of $H\otimes\bC$ given by the subspaces:
\begin{equation*}
I^{p,q}=F^p\cap W_{p+q}\cap \left(\ol{F^q}\cap W_{p+q}+\sum_{j\ge 2} \ol{F^{q-j+1}}\cap W_{p+q-j}\right)
\end{equation*}
satisfy the following condition:
\begin{equation}\label{eq-direct}
W_k=\bigoplus_{p+q\le k}I^{p,q}, \quad F^r=\bigoplus_{p\ge r}I^{p,q}. 
\end{equation}
\end{thm}
Applying this splitting result to the data from Theorem \ref{thm-LMH}, we get:
\begin{lem}
With the notation from above, we have the following identities:
\begin{eqnarray*}
&&I^{1,1}=W_2=\ol{I^{1,1}}; \\
&&I^{1,2}=F^1\cap \ol{F^2}\cap W_3, \quad I^{2,1}=F^2\cap \ol{F^1}\cap W_3=\ol{I^{1,2}}; \\
&&I^{3,0}=F^3, \quad I^{0,3}=\ol{F^3}=\ol{I^{3,0}}\\
&&I^{2,2}=F^2\cap (\ol{F^2}+W_2).  
\end{eqnarray*}
In particular, $\dim_\bC I^{1,1}=1=\dim_\bC I^{2,2}$, $\dim_\bC I^{1,2}=h$, $\dim_\bC I^{3,0}=1$ and
$I^{p,q}=0$ otherwise. 
As a consequence, we have the following decompositions for the Hodge filtration:
\begin{equation}\label{eq-HFsplitting}
F^1=I^{1,1}\oplus I^{2,1}\oplus I^{1,2} \oplus I^{2,2}\oplus I^{3,0}, \quad F^2=I^{2,1}\oplus I^{2,2} \oplus I^{3,0}, \quad F^3=I^{3,0}
\end{equation}
and for the weight filtration:
\begin{equation}\label{eq-WFsplitting}
W_2=I^{1,1}, \quad W_3=W_2\oplus I^{1,2}\oplus I^{2,1}\oplus I^{3,0}\oplus I^{0,3}, \quad W_4=W_3\oplus I^{2,2}.  
\end{equation}
\end{lem}
\begin{proof}
Since the calculation is quite straightforward, we only show the identity for $I^{1,1}$ and $I^{2,2}$ and leave the verification of other identities to the reader:
\begin{eqnarray*}
I^{1,1}&=&F^1\cap W_2 \cap \left(\ol{F^1}\cap W_2+\ol{F^0}\cap W_0\right)=W_2; \\
I^{2,2}&=&F^2\cap W_4 \cap \left(\ol{F^2}\cap W_4+\ol{F^1}\cap W_2+\ol{F^0}\cap W_1\right)\\
&=&F^2\cap (\ol{F^2}+W_2)
\end{eqnarray*}
because $F^1\cap W_2=W_2=\ol{F^1}\cap W_2$ by Theorem \ref{thm-LMH}.(iii).(a). 
\end{proof}
Combing this decomposition with Theorem \ref{thm-LMH}.(ii), we know that the induced Hodge decomposition $W_3/W_2\cong I^{2,1}\oplus I^{1,2}\oplus I^{3,0}\oplus I^{0,3}$ corresponds to the Hodge decomposition on $H^3(\tilde{X})\cong H^3(X)$ under the isomorphism $W_3/W_2\cong H^3(\tilde{X})$. 

Let $\la \cdot, \cdot\ra$ denote the non-degenerate intersection form on $H:=H^3_{\lim}\cong H^3(Y_s)$ (with $s\in S^*$) which is defined over $\bQ$.
Fix a non-zero rational vector $e_{h+1}\in I^{1,1}=W_2$.  Denote 
\begin{equation}\label{eq-realV}
V=I^{3,0}\oplus I^{0,3}\oplus I^{2,1}\oplus I^{1,2}.  
\end{equation}
Then $V$ is a sub-vector space of $H$ of dimension $2h+2$ with a real structure. It satisfies 
\begin{equation}
V\oplus I^{1,1}=V\oplus W_2=W_3, \quad V\cong W_3/W_2\cong H^3(\tilde{X}). 
\end{equation} 
Note that $(V\oplus I^{1,1})^\perp=W_3^\perp=W_2$ has dimension 1 and already contains $I^{1,1}$. It follows that 
 the restriction $\left.\la \cdot, \cdot\ra\right|_{V}$ is non-degenerate.  Otherwise $\dim(W_3^\perp \cap W_3)\ge 2$. 
Choose a symplectic base $\{e_0, \dots, e_{h}, f_0, \dots, f_{h}\}$ for $(V, \left.\la\cdot, \cdot\ra\right|_V)$. Define $f_{h+1}\in H_\bQ$ to be the vector that satisfies:
\begin{equation}\label{eq-inner1}
\la e_{h+1}, f_{h+1}\ra=1, \quad \text{ and}\quad  \la V, f_{h+1}\ra=0=\la f_{h+1}, f_{h+1}\ra. 
\end{equation}
Then $f_{h+1}\not\in W_3$ and hence $f_{h+1}+W_3=W_4$. 
Moreover by Theorem \ref{thm-LMH}.(iv), $N(f_{h+1})=k\cdot e_{h+1}$ with $k\neq 0\in \bQ$.


\section{Proof of Theorem \ref{thm-HR}}

We continue to use the notation from above. In particular, we choose the basis vector $\{e_0, e_1, \dots, e_{h+1}; f_0, f_1, \dots, f_{h+1}\}$ of $H=H^3(Y_s;\bC)$ as in the end of last section.

Recall from \eqref{eq-HFsplitting} that $F^2=I^{3,0}\oplus I^{2,1}\oplus  I^{2,2}$. We choose a complex basis for $F^2$, by first 
choosing a nonzero vector $u_0\in I^{3,0}$ and a complex basis $\{u_1,\dots, u_h\}$ of $I^{2,1}$. 
Because $I^{2,2}+W_3=W_4$ (\eqref{eq-WFsplitting}) and $f_{h+1}+ W_3=W_4$, we know that
there is a nonzero vector in 
$I^{2,2}$ whose coefficient in front of $f_{h+1}$ is equal to $1$:
\begin{equation}\label{eq-uh+1}
u_{h+1}=\sum_{i=0}^{h+1} {a}_i e_i+\sum_{j=0}^{h} {b}_j f_j +f_{h+1}
\end{equation}
where with $a_i, b_j$ are possibly complex numbers. So we get a basis $\{u_0, u_1,\dots, u_h, u_{h+1}\}$ for $F^2$. 

For simplicity of notation, we denote $F^p=F^p_{\lim}$, and for $s\in S^*$, denote $F^p_s=F^p H^3(Y_s; \bC)$. 
Let us first prove the equality in \eqref{eq-opposed}: $F^2_s\cap \ol{F^2_s}=0$. 
There are two cases for the subspace $F^2$ of $H$:
\vskip 3mm
Case \textbf{(I)}: $F^2\cap \ol{F^2}=0$; Case \textbf{(II)}: $F^2\cap \ol{F^2} \neq 0$. 
\vskip 3mm

In case \textbf{(I)}, we immediately get $F^2_s\cap \ol{F^2_s}=0$ for $s$ sufficiently close to $0$. By the discussion in the introduction before Theorem \ref{thm-HR}, we also know that the Hodge filtration on $H^3(Y_s; \bC)$ induces a pure Hodge structure.  

Next we consider case \textbf{(II)}. Because $F^2\cap \ol{F^2}\subset F^2\cap (\ol{F^2}+W_2)=I^{2,2}$ and $\dim I^{2,2}=1$, we know that in fact $I^{2,2}=F^2\cap \ol{F^2}=\ol{I^{2,2}}$ is a one-dimensional subspace of $H$ with a real structure. 
Because $\cF^2$ is a holomorphic sub-vector bundle of $\mcH$, the basis $\{u_\alpha; \alpha=0, 1, \dots, h+1\}$ of $F^2=F^2_{\lim}$ chosen above extends to become a basis $\{u_\alpha(s); \alpha=0, 1, \dots, h+1\}$ of
$F^2_s$.  
For any $\alpha\in \{0,\dots, h+1\}$, we can write:
\begin{equation}\label{eq-decomp1}
u_\alpha(s)=\sum_{i=0}^h A_{\alpha,i}e_i+\sum_{i=0}^h B_{\alpha,i}f_i+C_{\alpha}e_{h+1}+D_{\alpha}f_{h+1}.
\end{equation}
For any $s\in S$, we will also use $s$ to denote its coordinates $(t, \zeta)$ with respect to the fixed isomorphism $S\cong \Delta^h\times \Delta$ and let $(t, z)\in $ denote the standard coordinates on $\widetilde{S^*}=\Delta^h\times \mathfrak{H}$. They are related by $\zeta=e^{2\pi \sqrt{-1} z}$ or $z=\frac{1}{2\pi \sqrt{-1}}\log \zeta$. The coefficients $A_{\alpha,i}, B_{\alpha,i}, C_\alpha, D_\alpha$ in \eqref{eq-decomp1} are holomorphic functions of $(t, \zeta)$. For convenience, we introduce the following notion.
\begin{defn}
With the above notation, we say that $Y_s$ is a $\delta$-small smoothing of $\bar{X}$ if $s\in S^*$ and $|s|<\delta$ where $|s|=(|t|^2+|\zeta|^2)^{1/2}$. 
\end{defn}
Now that precise meaning of Theorem \ref{thm-HR} is that we can find $\delta>0$ such that any $\delta$-small smoothing of $\bar{X}$ satisfies \eqref{eq-opposed} and \eqref{eq-HR3}. 
Note the relation
\begin{equation}\label{eq-izbd}
\iz= - \frac{1}{2\pi}\log |\zeta| \ge - \frac{1}{2\pi} \log |s|.
\end{equation}
So if $Y_s$ is $\delta$-small with $\delta\ll 1$, then $\iz>-\frac{1}{2\pi}\cdot \log \delta \gg 1 $. 

After restricting to $S^*$, we pullback these basis vectors to $\vphi^*\mcH|_{S^*}$ to get:
\begin{eqnarray}\label{eq-decomp2}
v_\alpha(t, z)&:=&\vphi^*u_\alpha(t, \zeta)=\sum_{i=0}^{h} \tilde{A}_{\alpha, i} e_i+\sum_{i=0}^h \tilde{B}_{\alpha, i} f_i+\tilde{C}_{\alpha} e_{h+1}+ \tilde{D}_{\alpha} f_{h+1}
\end{eqnarray}
where the coefficients are now holomorphic functions of $(t, z)$. 

By the discussion in the paragraph after Theorem \ref{thm-E1deg}, we have $v_\alpha(t, z)=e^{zN} u_\alpha(t, \zeta)=(1+z N) u_{\alpha}$ which gives us the identities: $\tilde{A}_{\alpha,i}(t, z)=A_{\alpha,i}(t, \zeta)$, $\tilde{B}_{\alpha,i}(t, z)=B_{\alpha,i}(t, \zeta)$, $\tilde{D}_\alpha(t, z)=D_\alpha(t, \zeta)$ and $\tilde{C}_{\alpha}(t, z)=z D_{\alpha}(t, \zeta)+C_\alpha(t, \zeta)$. 
So we have the identity (see \cite[3.3]{Fri20}):
\begin{eqnarray}\label{eq-decomp3}
v_\alpha(t, z)&=& \sum_{i=0}^{h} A_{\alpha,i} e_i+\sum_{i=0}^h B_{\alpha,i} f_i+(z D_{\alpha}+C_{\alpha}) e_{h+1}+ D_{\alpha} f_{h+1}\nonumber \\
&=:&v'_\alpha+\tilde{C}_{\alpha}e_{h+1}+D_{\alpha} f_{h+1},
\end{eqnarray}
where $v'_\alpha=\sum_{i=0}^h A_{\alpha,i}e_i+\sum_{i=0}^h B_{\alpha,i}f_i\in V=I^{3,0}\oplus I^{0,3}\oplus I^{2,1}\oplus I^{1,2}$ (see \eqref{eq-realV}). Note that $v_\alpha$ satisfies the property $v_\alpha(t, z+1)=(\mathrm{Id}+N) v_\alpha(t, z)$ such that the invariant section $e^{-zN}v_\alpha(t,z)$ descends to become the section $u_\alpha$. Because of such invariance property, in the following argument, we can assume that $z$ lies in a vertical strip of bounded width.

To prove that $F^2_s\cap \ol{F^2_s}\neq 0$, we need to show the non-vanishing:
\begin{equation}\label{eq-bigwedge}
\frac{\bigwedge_{\alpha=0}^{h+1} v_\alpha\wedge \ol{v_\alpha}}{e_{h+1}\wedge f_{h+1}\wedge \bigwedge_{i=0}^h e_i\wedge f_i
}\neq 0. 
\end{equation}

To get \eqref{eq-bigwedge}, we calculate:
\begin{eqnarray}\label{eq-vwedge}
v_\alpha\wedge \ol{v_\alpha}&=&\left(v'_\alpha+(z D_\alpha+C_\alpha) e_{h+1}+D_\alpha f_{h+1}\right)\wedge \left(\ol{v'_\alpha}+(\bar{z} \ol{D_\alpha}+\ol{C_\alpha}) e_{h+1}+\ol{D_\alpha} f_{h+1}\right)\nonumber \\
&=&v'_\alpha\wedge \ol{v'_\alpha}+
(\bar{z}\ol{D_\alpha}+\ol{C_\alpha}) v'_\alpha \wedge e_{h+1}+(z D_\alpha+C_\alpha)e_{h+1} \wedge \ol{v'_\alpha}\nonumber\\
&&+\ol{D_\alpha} v'_\alpha\wedge f_{h+1}+D_{\alpha} f_{h+1}\wedge \ol{v'_\alpha}\nonumber\\
&&+(z-\bar{z})|D_\alpha|^2 e_{h+1}\wedge f_{h+1}\nonumber\\
&&+(C_\alpha \ol{D_\alpha}-D_\alpha \ol{C_\alpha}) e_{h+1}\wedge f_{h+1}
\end{eqnarray}
According to the choice of basis vectors of $F^2$ at the beginning of this section, when $\alpha=i\in \{0,\dots, h\}$, $v'_i(0)=u_i \in I^{3,0}\oplus I^{2,1}\subset V$ which implies 
\begin{enumerate}[(i)]
\item $v_i=u_i+O(|s|)$ which means that the coefficients of $v_i-u_i$ are of order at most $O(|s|)$; 
\item $C_i(0)=0=D_i(0)$ which implies $C_i(s)=O(|s|)$ and $D_i(s)=O(|s|)$. 
\end{enumerate}
So for any $\alpha=i\in \{0, \dots, h\}$, we can then write:
\begin{equation}\label{eq-vwedge2}
v_i(t, z)\wedge \ol{v_i(t, z)}=u_i\wedge \ol{u_i}+\bar{z}\ol{D_i} v'_i \wedge e_{h+1}+z D_i e_{h+1}\wedge \ol{v'_i}+(z-\bar{z})|D_i|^2 e_{h+1}\wedge f_{h+1}
+O(|s|)
\end{equation}
where $O(|s|)$ consists of terms whose coefficients with respect to the induced basis of $\wedge^2 H$ (by the basis $\{e_\alpha, f_\alpha\}$) are of order at most $O(|s|)$.
On the other hand, when $\alpha=h+1$, then $u_{h+1}(0)=u_{h+1}$ from \eqref{eq-uh+1}, and the coefficients of $u_{h+1}(s)$ are all of order $O(1)$. We calculate by using \eqref{eq-vwedge} to get:
\begin{eqnarray}\label{eq-vwedge3}
v_{h+1}(t,z)\wedge \ol{v_{h+1}(t,z)}&=&(z-\bar{z})|D_{h+1}|^2 e_{h+1}\wedge f_{h+1}\nonumber \\
&&+\bar{z}\ol{D_{h+1}} v'_{h+1}\wedge e_{h+1}+z D_{h+1} e_{h+1}\wedge \ol{v'_{h+1}}+O(1), 
\end{eqnarray}
where $O(1)$ consists of terms whose coefficients are uniformly bounded for $s\in S$. 

Because in the top wedge product any basis vector can appear only once, 
we can use the expressions \eqref{eq-vwedge2} and \eqref{eq-vwedge3} to easily see that the left-hand-side of \eqref{eq-bigwedge} is given up to a nonzero constant by:
\begin{eqnarray}\label{eq-leading1}
&&(z-\bar{z}) |D_{h+1}|^2(1+O(|s|))+\bar{z}\ol{D_{h+1}} O(|s|)+z D_{h+1} O(|s|)\nonumber \\
&&\hskip 2cm +O(1) \cdot \bar{z} \ol{D_i}+O(1)\cdot z D_i+O(1) (z-\bar{z})|D_i|^2+O(1)\nonumber\\
&=& 2\sqrt{-1} \iz\cdot \left[1+O(|s|)\right]+O(1)
\end{eqnarray} 
where we used $D_{h+1}=1+O(|s|)$, $D_{i}=O(|s|)$ and hence $|z D_i|=\iz\cdot O(|s|)$ (recall that we can assume that $z$ lies in a vertical strip of bounded width in $\mathfrak{H}$). So, by also taking \eqref{eq-izbd} into account, we conclude that there exists $\delta>0$ such that that last quantity in \eqref{eq-leading1} (which is a nonzero multiple of the left-hand-side of \eqref{eq-bigwedge}) is indeed non-zero as long as $|s|<\delta$. 

\begin{rem}\label{rem-choice}
In \cite{Fri20}, Friedman calculated a similar wedge product for a different basis of $F^2$, which is obtained by differentiating a holomorphic frame $F^3$ with respect to the Gauss-Manin connection. This choice of basis depends on a special property of differential of the period map in the Calabi-Yau case (see \cite[2.7]{Fri20}). Moreover the coefficients in his basis are only meromorphic. 

Our choice of basis, based on Deligne's splitting, is more direct and does not use the special property of the period map in the Calabi-Yau case. Moreover the coefficients are holomorphic. 
The calculation above shows that this simpler choice of basis actually works better for the main purpose. See also Remark \ref{rem-1dim}. 
\end{rem}

Next we prove the Hodge-Riemann bilinear relation \eqref{eq-HR3}. 
For any $s\in S$, set $H_s^{2,1}:=F_s^2\cap \ol{F_s^1}$. 
Because the Hodge-de Rham spectral sequence degenerates at the $E_1$-page (by Theorem \ref{thm-E1deg}.(iii)) when $s\in S^*$ is sufficiently small, we know that
\begin{equation*}
H^1(Y_s, \Omega_{Y_s}^2)\cong F^2_s/F^3_s. 
\end{equation*}
We get $\dim H^1(Y_s, \Omega^2_{Y_s})=\dim F_s^2-\dim F_s^3=(h+2)-1=h+1$. 
Because we know that $F_s^3\oplus \ol{F_s^1}=H_s=H^3(Y_s)$, the natural map 
\begin{equation}
H_s^{2,1}=F_s^2\cap \ol{F_s^1}\longrightarrow F_s^2/F_s^3
\end{equation}
is an isomorphism. In particular, $H^{2,1}_s$ has constant rank $h+1$ for any $s\in S^*$. Because the dimension of $H^{2,1}_s$ is upper semicontinuous, we get for $s\in S\setminus S^*=D$, 
\begin{equation}\label{eq-dimup}
\dim H_s^{2,1}\ge h+1.
\end{equation}

On the other hand, we can determine the dimension of $H^{2,1}_0=F^2\cap \ol{F^1}=F^2_{\lim}\cap \ol{F_{\lim}^1}$:
\begin{lem}\label{lem-H12}
The intersection $H^{2,1}_0=F^2\cap \ol{F^1}$ has dimension $h+1$. 
\end{lem}
\begin{proof}
Consider the two cases \textbf{(I)} and \textbf{(II)} as before. 
In case \textbf{(II)}, $I^{2,2}=\ol{I^{2,2}}$ and by \eqref{eq-HFsplitting} we get:
$H^{2,1}_0=I^{2,1}\oplus I^{2,2}$ which has dimension $h+1$.  

Let us consider the case \textbf{(I)}. \eqref{eq-HFsplitting} tell us that any element $u\in F^2$ has a representation:
\begin{equation}\label{eq-inter1}
u=\kappa_0 u_0 +\sum_{i=1}^h  \kappa_i u_i +\epsilon  u_{h+1}\in I^{3,0}\oplus I^{2,1}\oplus I^{2,2}. 
\end{equation}
On the other hand by \eqref{eq-HFsplitting} this lies in $\ol{F^1}$ if there exists a decomposition:
\begin{eqnarray}\label{eq-inter2}
u&=&\sum_{i=1}^h \kappa'_i u_i+\theta'_0 \ol{u_0}+\sum_{i=1}^h \theta'_i \ol{u_i}+ \lambda' e_{h+1}+\epsilon' \ol{u_{h+1}}\\
&&\hskip 2cm \in I^{2,1}\oplus I^{0,3}\oplus I^{1,2}\oplus I^{1,1}\oplus \ol{I^{2,2}}. \nonumber 
\end{eqnarray}  

If $\epsilon=0$, then it is easy to see that $\epsilon'=\kappa_0=\theta'_0=\theta'_i=\lambda'=0$ and $\kappa_i=\kappa'_i$ so that $u=\sum_{i=1}^h \kappa_i {u_i} \in I^{2,1}$. 

Let us assume that $\epsilon\neq 0$. By rescaling, we can assume that $\epsilon=1$. We first write $u_{h+1}$ from \eqref{eq-uh+1} into a different representation:
\begin{equation*}
u_{h+1}=\sum_{i=0}^{h} a'_i u_i+\sum_{i=0}^h \ol{b'_i} \ol{u}_i + c' e_{h+1} + f_{h+1}. 
\end{equation*}
Comparing the coefficients of \eqref{eq-inter1}-\eqref{eq-inter2} give us the following system of equations:
\begin{equation*}
\epsilon'=\epsilon=1, \quad 
\kappa_0+a'_0={b'_0}, \quad \kappa_i+a'_i=\kappa'_i+b'_i, \quad \ol{b'_0}=\theta'_0+\ol{a'_0}, \quad
\ol{b'_i}=\theta'_i+\ol{a'_i}, \quad c'=\lambda'+\ol{c'},
\end{equation*}
which means that the following coefficients are uniquely determined:
\begin{equation*}
\kappa_0={b'_0}-{a'_0}, \quad \theta'_0=\ol{b'_0}-\ol{a'_0}, \quad \theta'_i=\ol{b'_i}-\ol{a'_i}, \quad \lambda'=c'-\ol{c'}
\end{equation*}
and $\kappa'_i=\kappa_i+a'_i-b'_i$ while $\{\kappa_i; i=1, \dots, h\}$ are free variables. 
Set 
\begin{equation}\label{eq-u*h+1}
u^*_{h+1}=\kappa_0 u_0+u_{h+1}={b'_0}u_0+\sum_{j=1}^ha'_j u_j+\sum_{i=0}^h \ol{b'_i} \ol{u_i}+c' e_{h+1}+f_{h+1}.
\end{equation} 
Then we get:
\begin{equation}
u=u^*_{h+1}+\sum_{i=0}^h \kappa_i {u_i}\in u^*_{h+1}+I^{2,1}.
\end{equation}
So conclude that in case \textbf{(I)}, $H^{2,1}_0=\bC u^*_{h+1}+I^{2,1}$ again has dimension $h+1$. 

\end{proof}
  
This lemma tells us that $\dim H_0^{2,1}=h+1=3h+5-\dim (F^1_s+\ol{F^2_s})$ which implies that 
\begin{equation}\label{eq-dimdown}
\dim H_s^{2,1}=(3h+5)-\dim (F^1_s+\ol{F^2_s}) \le h+1
\end{equation} 
for $s$ sufficiently small. 
Combining this with \eqref{eq-dimup} we get $\dim H^{2,1}_s\equiv h+1$ for any $s\in S$ sufficiently small.  

So after possible shrinking, we conclude that $\{H^{2,1}_s; s\in S\}$ form a smooth sub-vector bundle $\mcH^{2,1}$ of $\mcH$ by the following easy lemma, whose proof we leave to the reader. \footnote{There is a proof in \href{https://mathoverflow.net/questions/85407/intersection-of-subvector-bundles}{https://mathoverflow.net/questions/85407/intersection-of-subvector-bundles}.}

\begin{lem}
Let $\mcH\cong S\times \bC^d$ be a trivial vector bundle over $S=\Delta^h\times \Delta$. Let $\mcH'$ and $\mcH''$ be two smooth  sub-vector bundle of $\mcH$ of ranks $d_1$ and $d_2$ respectively. If $\dim (\mcH'_s \cap \mcH''_s)$ is constant for any $s\in S$, then $\mcH'_s\cap \mcH''_s$ is a smooth sub-vector bundle of $\mcH$ of rank $\dim (\mcH'_s\cap \mcH''_s)$.  
\end{lem}

By the proof of Lemma \ref{lem-H12}, we can choose basis $\{u_1, \dots, u_h, u_{h+1}\}$ such that $\{u_1, \dots, u_h\}$ is a basis of $I^{2,1}$ and $u_{h+1}$ satisfies
$
u_{h+1}+I^{2,1}=H^{2,1}
$ and the coefficient of $u_{h+1}$ in front of $f_{h+1}$ is equal to 1. Indeed we can choose $u_{h+1}$ to be the same as the vector from \eqref{eq-uh+1} in case \textbf{(II)} and choose $u_{h+1}$ to be the vector $u^*_{h+1}$ from \eqref{eq-u*h+1} in case \textbf{(I)} (so that we change the notation from $u^*_{h+1}$ to $u_{h+1}$). 

Because $\mcH^{2,1}$ is a smooth vector bundle, we can then extend it to be become a basis $\{u_1(s), \dots, u_h(s), u_{h+1}(s)\}$ of $H_s^{2,1}=F_s^2\cap \ol{F_s^1}$ as long as $s$ sufficiently small. Each $u_i(s)$ is in general only smooth $H$-valued function of $(t, \zeta)$. 
For any $\alpha \in \{1, \dots, h+1\}$, we can still define $v_\alpha=\vphi^*u_\alpha$ and get decomposition as in \eqref{eq-decomp1}-\eqref{eq-decomp3}. But now the coefficients $A_{\alpha, i}, B_{\alpha, i}, C_\alpha, D_\alpha$ are only known to be smooth functions of $(t, \zeta)$. 

Set $Q(\cdot, \cdot)=-\sqrt{-1}\la \cdot, \cdot\ra$ and
$Q_{\alpha\beta}=Q(v_\alpha, \ol{v_\beta})=-\sqrt{-1}\la v_\alpha, \ol{v_\beta}\ra$ for $\alpha, \beta\in \{1,\dots, h+1\}$. We need to show that the matrix $\{Q_{\alpha\beta}\}$ is positive definite when $s\in S^*$ is sufficiently small (and hence $\iz$ is sufficiently big by \eqref{eq-izbd}).
We calculate:
\begin{eqnarray}\label{eq-Qmatrix}
Q_{\alpha\beta}&=&Q(v'_\alpha+ (z D_{\alpha}+C_\alpha) e_{h+1}+D_{\alpha} f_{h+1}, \ol{v'_\beta}+(\bar{z} \ol{D_{\alpha}}+\ol{C_\alpha}) e_{h+1}+\ol{D_{\beta}} f_{h+1})\nonumber \\
&=&Q(v'_\alpha, \ol{v'_\beta})-\sqrt{-1}(z-\bar{z}) D_{\alpha}\ol{D_{\beta}}+\sqrt{-1}(-C_\alpha \ol{D_{\beta}}+D_{\alpha}\ol{C_{\beta}})
\end{eqnarray} 
where we used the fact that $Q(V, e_{h+1})=0=Q(V, f_{h+1})$ and $Q(e_{h+1}, f_{h+1})=-\sqrt{-1}$, according to the choice of the symplectic basis $\{e_0, \dots, e_{h+1}; f_0, \dots, f_{h+1}\}$ (see \eqref{eq-inner1}). 

We consider $\{Q_{\alpha\beta}\}$ as the block matrix 
\begin{equation*}
\{Q_{\alpha\beta}\}=
\left(
\begin{array}{cc}
\{\hat{Q}_{ij}\} & \{Q_{i,h+1}\}\\
\{Q_{h+1,j}\} & Q_{h+1, h+1}
\end{array}
\right)
=:
\left(
\begin{array}{cc}
\hat{Q} & \phi \\
\phi^* & \rho
\end{array}
\right)
\end{equation*}
where for simplicity of notation, we have set $\rho=Q_{h+1,h+1}=Q(v_{h+1},v_{h+1})$ and $\phi=\{\phi_i\}:=\{Q_{i,h+1}=Q(v_i, \ol{v_{h+1}})\}_{1\le i\le h}$.

To prove the positive definiteness of $\{Q_{\alpha\beta}\}$, we first estimate $\rho=Q_{h+1,h+1}$ by using \eqref{eq-Qmatrix} to get:
\begin{equation*}
\rho=Q(v'_{h+1}, \ol{v'_{h+1}})-\sqrt{-1}(z-\bar{z})D_{h+1}\ol{D_{h+1}}+\sqrt{-1}(-C_{h+1}\ol{D_{h+1}}+D_{h+1}\ol{C_{h+1}})
\end{equation*}
Because all coefficients $A_{h+1,i}$, $B_{h+1,i}$, $C_{h+1}$ and $D_{h+1}$ as functions of $s=(t, \zeta)$ are of order at most $O(1)$, we get:
\begin{equation*}
Q(v'_{h+1}, \ol{v'_{h+1}})=O(1), \quad \left|\sqrt{-1}(C_{h+1}\ol{D_{h+1}}-D_{h+1}\ol{C_{h+1}})\right|=O(1).
\end{equation*}
Because $D_{h+1}(0)=1$ we have $D_{h+1}(s)=1+O(|s|)$ and hence:
\begin{equation}\label{eq-rhoest}
\rho=Q_{h+1,h+1}=2 \iz (1+O(|s|))+O(1) 
\end{equation}
which is positive as long as $|s|\ll 1$ (and hence $|\iz|\gg 1$ by \eqref{eq-izbd}).  

Before we continue, note that, because, for any $i\in \{1,\dots, h\}$ $u_i\in I^{2,1}\subset V_\bC=\mathrm{span}_\bC\{e_0,\dots, e_{h}, f_0, \dots, f_h\}$, we have the vanishing $C_i(0)=D_i(0)=0$ which implies 
\begin{equation}\label{eq-CiDi}
|C_i(s)|=O(|s|), \quad |D_i(s)|=O(|s|).
\end{equation} 
Next we do an elimination:
\begin{equation}
T Q T^*=\left(
\begin{array}{cc}
\hat{Q}-\rho^{-1}\phi\phi^*&0\\
0& \rho
\end{array}
\right) 
\quad \text{ with }\quad  
T=\left(
\begin{array}{cc}
\mathrm{Id}_{h}&  -\rho^{-1}\phi\\
0& 1
\end{array}
\right)
\end{equation}
where $\mathrm{Id}_h$ is the identity matrix of size $h\times h$.
Since we already know that $\rho>0$ as long as $|s|\ll 1$,
it is clear that, to prove the positive definiteness of $Q$, it suffices to prove that the sub-matrix $P:= \hat{Q}-\rho^{-1}\phi\phi^*$ is positive definite. 
For simplicity of notation, denote $y=\iz>0$. Let us calculate by using \eqref{eq-Qmatrix}:
\begin{eqnarray*}
\rho \cdot P_{ij}&=&\rho Q_{ij}-\phi_i \ol{\phi_j}=Q(v_{h+1}, \ol{v_{h+1}})Q(v_i, \ol{v_j})-Q(v_i, \ol{v_{h+1}})Q(v_{h+1}, \ol{v_j})\\
&=&\left(Q(v'_{h+1}, \ol{v'_{h+1}})+2 y |D_{h+1}|^2+\sqrt{-1}(-C_{h+1}\ol{D_{h+1}}+D_{h+1}\ol{C_{h+1}}\right)\\
&&\cdot \left(Q(v'_i, \ol{v'_j})+ 2y D_i\ol{D_j}+\sqrt{-1}(-C_i \ol{D_j}+D_i\ol{C_j})\right)\\
&&-\left(Q(v'_i, \ol{v'_{h+1}})+2y D_i \ol{D_{h+1}}+\sqrt{-1}(-C_i \ol{D_{h+1}}+D_i\ol{C_{h+1}})\right)\\
&&\hskip 5mm\cdot \left(Q(v'_{h+1}, \ol{v'_j})+2 y D_{h+1} \ol{D_j}+\sqrt{-1}(-C_{h+1}\ol{D_j}+D_{h+1}\ol{C_j})\right)\\
&=&2 y |D_{h+1}|^2Q(v'_i, \ol{v'_j})+2y |D_{h+1}|^2 \sqrt{-1} (-C_i\ol{D_j}+D_i\ol{C_j})\\
&&+2 y Q(v'_{h+1}, \ol{v'_{h+1}}) D_i \ol{D_j}+2 y D_i \ol{D_j}\sqrt{-1}(-C_{h+1}\ol{D_{h+1}}+D_{h+1}\ol{C_{h+1}})\\
&&-2y D_i \ol{D_{h+1}}Q(v'_{h+1}, \ol{v'_j})-2y D_i \ol{D_{h+1}}\sqrt{-1}(-C_{h+1}\ol{D_j}+D_{h+1}\ol{C_j})\\
&&-2y D_{h+1}\ol{D_j}Q(v'_i, \ol{v'_{h+1}})-2y D_{h+1}\ol{D_j}\sqrt{-1}(-C_i \ol{D_{h+1}}+D_i \ol{C_{h+1}})\\
&&+O(1).
\end{eqnarray*}

The points of this calculation are that the $y^2$-term get (surprisingly) cancelled, and there is a (leading) term $2 y|D_{h+1}|^2 Q(v'_i, \ol{v'_j})$ while the other $y$-terms all contain at least one factor from $\{C_i, D_i, C_j, D_j\}$ with $i,j\in \{1,\dots, h\}$ and hence can be written as $y\cdot O(|s|)$ thanks to \eqref{eq-CiDi}. 
So we can write
\begin{eqnarray}
P_{ij}&=&\rho^{-1}\cdot \rho P_{ij}=\frac{1}{2 y (1+O(s))+O(1)}\left(2y |D_{h+1}|^2 Q(v'_i, \ol{v'_j})+y\cdot O(|s|)+O(1)\right)\nonumber \\
&=&|D_{h+1}|^2 Q(v'_i, \ol{v'_j})+O(|s|)+O(y^{-1}).
\end{eqnarray}
Because both terms $O(|s|)$ and $O(y^{-1})$ are negligible if $|s|\ll 1$ is sufficiently small (see \eqref{eq-izbd}), to prove that $\{P_{ij}\}$ is positive definite, we just need to prove that $\{Q(v'_i, \ol{v'_j})\}$ is uniformly positive definite by which we mean that it is positive definite and is bounded from below by $\epsilon\cdot \mathrm{Id}_h$ for some $\epsilon>0$ independent of $s$. 
\begin{rem}\label{rem-1dim}
If $\dim S=1$ so that $s=\zeta$ and $y=\iz= -\frac{1}{2\pi}\log |s|$, then the calculations/estimates can be greatly simplified (as shown in the first arXiv-version of this paper). 
\end{rem}
Because $v'_i$ are small perturbations of $u_i$ (see the property (i) after \eqref{eq-vwedge}), we see that $\{Q(v'_i, \ol{v'_j})\}$ is a small perturbation of the $h\times h$-matrix 
$\{Q'_{ij}\}=\{Q(u_i, \ol{u_j})\}_{i,j\in \{1,\dots,h\}}$ which is just matrix of the restriction $Q|_{I^{2,1}}$ with respect to the basis $\{u_i; i=1,\dots, h\}$ of $I^{2,1}$. 
So it suffices to show that $\{Q'_{ij}\}$ is positive definite. 

To prove this positivity, we first remark that $Q=-\sqrt{-1}\la \cdot, \cdot\ra$ naturally induces a bilinear form $\bar{Q}$ on $W_3/W_2$. In other words, for any $[u], [v]\in W_3/W_2$, let $u, v\in W_3$ be their liftings and set:
\begin{equation}
\bar{Q}([u], [v])=Q(u, v)=-\sqrt{-1}\la u, v\ra.
\end{equation}
Because $W_2=W_3^\perp$, the right-hand-side does not depend on the liftings of $[u], [v]$. Now according to the splitting \eqref{eq-WFsplitting}, there is a canonical identification of $V=I^{3,0}\oplus I^{0,3}\oplus I^{2,1}\oplus I^{1,2}$ with $W_3/W_2$ which preserves the Hodge decomposition. 
It is immediate from the defintion that under this identification the bilinear form $Q_V(\cdot, \cdot):=-\sqrt{-1}\la\cdot, \cdot\ra|_V$ is nothing but the the induced bilinear form on $W_3/W_2$. Now the key is the following lemma.  

\begin{lem}\label{lem-transit}
Under the natural isomorphism of Hodge structures $W_3/W_2\cong H^3(\tilde{X})$, the bilinear form $\bar{Q}(\cdot,\cdot)$ is isomorphic to the bilinear form $-\sqrt{-1}\la \cdot, \cdot\ra_{\tilde{X}}$ on $H^3(\tilde{X})$. 
\end{lem}
This lemma would complete the proof of Theorem \ref{thm-HR}. Indeed, with the identification $(V, Q_V(\cdot, \cdot))\cong (W_3/W_2, \bar{Q}(\cdot, \cdot))$, 
we know that the matrix $Q'=\{Q(u_i, \ol{u_j})\}=\{Q_V(u_i, \ol{u_j})\}$ is indeed positive definite by the second Hodge-Riemann bilinear relation for $\tilde{X}$. 

Finally we only need to prove Lemma \ref{lem-transit}.
\begin{proof}[Proof of Lemma \ref{lem-transit}]
We first note that the morphisms of mixed Hodge structures from Theorem \ref{thm-LMH}.(i)-(ii) are induced by natural geometric maps.
Indeed, the quotient morphism $H^3(Y_0)\rightarrow H^0(\tilde{X})$ is simply induced by the closed embedding $\tilde{X}\rightarrow Y_0$.
On the other hand, the injective morphism $\iota: H^3(Y_0)\cong H^3(\cY) \rightarrow H^3_{\lim}\cong H^3(Y_s)$ (with $s\in S^*$) is induced by the composition $Y_s \rightarrow \cY^*\rightarrow \cY$ and we used the fact that $Y_0\hookrightarrow \cY$ is a homotopy equivalence thus inducing an isomorphism $H^3(Y_0)\cong H^3(\cY)$. 

This description motivates our method to verify the statement of the lemma.
First we will lift any two elements $[u], [u'] \in H^3(\tilde{X})=W_3/W_2$ to elements in $H^3(\cY)\cong H^3(Y_0)=W_3$, and then restrict to get $u, u'\in H^3(Y_s)$. After this lifting, we will verify that the cup product of $u, u'$ on $Y_s$ agree with the cup product of $[u]$ and $[u']$ on $\tilde{X}$. 

This is indeed not difficult to achieve in by using de Rham cohomology. We first represent $[u]$ by a closed differential form $\psi$. Let $U$ be the disjoint union of open neighborhoods of exceptional divisors $E_i\cong \bP^1\times\bP^1$ in $\tilde{X}$. Then $U$ is homotopic to a disjoint union of $\bP^1\times\bP^1$. As a consequence, $H^3(U;\bC)=0$, and hence the closed form $\left.\psi\right|_{U}$ is exact. So there exists a smooth 2-form $\theta$ on $U$ such that $\psi|_{U}=d\theta$. Let $\eta$ be a cut-off function supported on $U$ and is identically equal to 1 on a smaller open neighborhood $U_1$ of $E_i$ such that $\ol{U_1}$ is a compact subset of $U$.
Then the differential form $\hat{\psi}=\psi-d (\eta\theta)$ vanishes on $U_1$ and represents $[u]$. 
We now want to extend $\hat{\psi}$ to get a closed two form $\hat{\Psi}$ on $\cY$ such that $\left.\hat{\Psi}\right|_{\tilde{X}_0}=\hat{\psi}$. 
For simplicity of notation, set $\ol{U_1}^c:=\tilde{X}\setminus \ol{U_1}$ so that $\hat{\psi}$ is supported on $\ol{U_1}^c$.  
By possibly shrinking $S$, we can then construct a deformation retraction $r: \cY\rightarrow Y_0$ such that $r^{-1}(\ol{U_1}^c)\cong \ol{U_1}^c\times S$ and the retraction is given by the projection. Then we define $\hat{\Psi}=r^*\hat{\psi}$ on $r^{-1}(\ol{U_1}^c)$ and is equal to 0 elsewhere. Then $\hat{\Psi}|_{Y_s}$ represents $u\in \iota(H^3(\cY))\subset H^3(Y_s)=H_{\lim}$. 

If $[u']$ is another element of $H^3(X; \bC)$, we can use the same process to get a lifting in $H^3(\cY; \bC)$ that is represented by a closed differential form $\hat{\Psi}'$ on $\cY$. So $u'\in \iota(H^3(\cY))\subset H^3(Y_s)=H_{\lim}$ is represented by $\hat{\Psi}'|_{Y_s}$.  

Then by the change of variable formula for integrals, we get the wanted identity:
\begin{equation*}
\la [u], [u']\ra_{\tilde{X}}= \int_{\tilde{X}} \psi\wedge \psi'=\int_{Y_s} \hat{\Psi}|_{Y_s} \wedge \hat{\Psi}'|_{Y_s}=\la u, u' \ra_{H_{\lim}}. 
\end{equation*}
\end{proof}

\vskip 3mm
\noindent
Department of Mathematics, Rutgers University, Piscataway, NJ 08854-8019.

\noindent
{\it E-mail address:} chi.li@rutgers.edu

\end{document}